\documentclass[10pt,reqno]{amsart}

\usepackage{amssymb, verbatim, mathrsfs}
\usepackage{amsmath}
\usepackage{amscd}
\usepackage[capitalise]{cleveref}
\usepackage{csquotes} 
\usepackage{enumitem}

\theoremstyle{plain}
\newtheorem{thm}{Theorem}[section]
\newtheorem*{thm*}{Theorem}
\newtheorem{prop}[thm]{Proposition}

\newtheorem{lem}[thm]{Lemma}

\theoremstyle{definition}
\newtheorem{defn}[thm]{Definition}
\newtheorem{eg}[thm]{Example}
\newtheorem{rem}[thm]{Remark}
\newtheorem{rems}[thm]{Remarks}
\newtheorem*{rem*}{Remark}
\newtheorem*{rems*}{Remarks}
\newtheorem*{note*}{Note}

\def\<{\langle}
\def\>{\rangle}
\def\geqs{\geqslant}
\def\leqs{\leqslant}
\def\ssm{\smallsetminus}
\def\inv{^{-1}}

\def\a{\alpha}
\def\b{\beta}
\def\g{\gamma}

\def\e{\epsilon}

\def\s{\sigma}

\def\v{\varphi}
\def\w{\omega}

\def\z{\zeta}
\def\upchi{\raise.4ex\hbox{$\chi$}}
\def\E{\mathcal E}
\def\R{\mathbb R}
\def\CC{\mathscr C}
\def\GG{\mathscr G}
\def\/{\kern 0.05em}

\newcommand{\ts}{\textstyle}
\newcommand{\nab}[1]{\nabla\kern-.2em\lower.8ex\hbox{\SMALL$#1$}\/}
\newcommand{\nabv}[1]{\nabla^v\kern-.3em\lower.8ex\hbox{\SMALL$#1$}\/}
\newcommand{\barnab}[1]{\bar\nabla\kern-.2em\lower.8ex\hbox{\SMALL$#1$}\/}
\newcommand{\nabsq}[2]{\nabla^2\kern-.3em\lower.8ex\hbox{\SMALL$#1,#2$}\/}
\newcommand{\shape}[1]{\A\kern-.1em\lower.8ex\hbox{\SMALL$#1$}\/}
\newcommand{\tr}[1]{\trace\kern-.2em\lower.8ex\hbox{\SMALL$#1$}\,}
\newcommand{\diffop}[2]{#1\kern-.1em\lower.8ex\hbox{\SMALL$#2$}\/} 
\newcommand{\dotprod}{\/\raise.25ex\hbox{\SMALL$\bullet$}\/}
\newcommand{\uperp}[1]{#1\raise1.1ex\hbox{\SMALL$\bot$}}

\DeclareMathOperator{\trace}{trace} 
\DeclareMathOperator{\grad}{grad}

\DeclareMathOperator{\volume}{vol}
\DeclareMathOperator{\Volume}{Vol}  
\DeclareMathOperator{\Ric}{Ric} 
\DeclareMathOperator{\diverge}{div}

\numberwithin{equation}{section}
\allowdisplaybreaks

\begin{document}

\title{Harmonic vector fields on pseudo-Riemannian manifolds}

\author{R.~M. Friswell}

\author{C.~M. Wood}
\address{Department of Mathematics \\
University of York \\
Heslington, York Y010 5DD\\
U.K.} \email{me@robertfriswell.co.uk, chris.wood@york.ac.uk}



\date{\today}

\begin{abstract}
The theory of harmonic vector fields on Riemannian manifolds is generalised to pseudo-Riemannian manifolds.  Harmonic conformal gradient fields on pseudo-Euclidean hyperquadrics are classified up to congruence, as are harmonic Killing fields on pseudo-Riemannian quadrics.  A para-K\"ahler twisted anti-isometry is used to correlate harmonic vector fields on the quadrics of neutral signature.
\end{abstract}

\maketitle
\thispagestyle{empty}

\section{Introduction}
Attempts to apply the variational theory of harmonic maps \cite{ES} to vector fields on Riemannian manifolds
%, as a criterion for optimality, 
foundered at an early stage when it was observed that, for a compact Riemannian manifold $(M,g)$, and with respect to the most natural metric $h$ on the total  space $TM$ of the tangent bundle (viz.\ the Sasaki metric \cite{S}), a vector field that is a harmonic map $(M,g)\to(TM,h)$ is necessarily parallel \cite{I,N}.  Moreover this remains the case if the vector field is only required to be a harmonic section of the tangent bundle \cite{W}.  A more interesting theory \cite{G} emerges in the special case where the vector field has constant length and is required to be a harmonic section of the corresponding isometrically embedded sphere sub-bundle of $TM$.  However this theory is necessarily limited, in the compact case, to manifolds of zero Euler characteristic.  
%In each of these situations, the Riemannian metric $h$ on the codomain---$TM$ or the total space of a sphere sub-bundle---is the Sasaki metric \cite{S} or its restriction, which despite being the most natural choice is ultimately responsible for the rigid behaviour of the energy functional.  
Thus, the prospects for using ``harmonicity'' as a criterion for optimality of vector fields, or more generally sections of Riemannian vector bundles, appeared limited.

\par
In \cite{BLW1} it was proposed to alleviate this problem by considering a wider range of metrics on $TM$.  More precisely, for a fixed Riemannian metric $g$ on $M$, there is an associated 2-parameter family $\CC\GG$ of {\sl generalised Cheeger-Gromoll metrics\/} on  $TM$:
\begin{equation*}
\CC\GG=\{h_{p,q}:p,q\in\R\},
\label{eq:gcg}
\end{equation*}    
in which $h_{0,0}=h$ (the Sasaki metric), $h_{1,1}$ is the Cheeger-Gromoll metric \cite{CG}, and $h_{2,0}$ is the stereographic metric; the general definition of $h_{p,q}$ is given in \eqref{eq:hpq} below.  The family $\CC\GG$ is ``natural'' in the sense of \cite{KS}, and more significantly renders the bundle projection $TM\to M$ a Riemannian submersion with totally geodesic fibres.  Furthermore, with only two degrees of freedom, $\CC\GG$ is a very tightly controlled deformation of the Sasaki metric.  It should be emphasised that this deformation has no affect on the Riemannian metric $g$ on $M$;
only the induced geometry of the tangent spaces varies.

\par
It turns out \cite{BLW1} that the energy functional behaves no less rigidly when the Sasaki metric $h_{0,0}$ is replaced by $h_{1,1}$ or $h_{2,0}$; however, other members of $\CC\GG$ permit greater flexibility.  In \cite{BLW2}, a {\sl harmonic vector field\/} on the Riemannian manifold $(M,g)$ was defined to be a harmonic section of $TM$ with respect to the Riemannian metric $g$ on $M$ and {\it some\/} $h_{p,q}\in\CC\GG$; classifications of harmonic vector fields were then obtained for conformal gradient fields and Killing fields on non-flat Riemannian space forms.  Typically (but not invariably) a harmonic vector field is metrically unique; that is, it picks a {\it unique\/} $h_{p,q}\in\CC\GG$.  Furthermore this $h_{p,q}$ has $q<0$,  which means that unlike the Sasaki, Cheeger-Gromoll and stereographic metrics, its signature varies across $TM$: Riemannian on a tubular neighbourhood of the zero section,  Lorentzian on the exterior of the tube, with a mild degeneracy on the boundary, a sphere bundle of radius $1/\sqrt{-q}$.  

\par
In view of the pseudo-Riemannian character of many elements of $\CC\GG$, in this paper we seek a generalisation of the theory of harmonic vector fields to pseudo-Riemannian manifolds (also referred to as semi-Riemannian manifolds \cite{ON}).  An immediate issue is that when the base metric $g$ is not Riemannian the Cheeger-Gromoll metric (ie.\ $h_{1,1}$) itself develops a codimension-one singularity, and this phenomenon persists for many other  $h_{p,q}\in\CC\GG$ (\cref{sec:GCGM}).  Thus, even when $M$ is compact, the energy functional for vector fields is not in general globally defined, so the variational problem under consideration is of necessity entirely local.  Despite this, and somewhat remarkably, the singularity in the energy functional is completely resolved at the level of the first variation: the Euler-Lagrange equations for harmonic sections with respect to any $h_{p,q}\in\CC\GG$ are in fact globally defined, and coincide (tensorially) with those in the Riemannian case (\cref{sec:HS}).  This enables us  (\cref{sec:HCGF}) to extend the classification of harmonic conformal gradient fields on Riemannian space forms obtained in \cite{BLW2} to hyperquadrics of pseudo-Euclidean space (\cref{thm:cgf}), and then (\cref{sec:HKF}) examine Killing fields on these spaces.  In particular, we obtain a condition for a preharmonic Killing field to be harmonic (\cref{thm:kf}).  (The notion of preharmonicity was introduced in \cite{BLW2}, and may be viewed as an integrability condition for harmonicity.)  We show (\cref{sec:2d}) that all Killing fields on the $2$-dimensional pseudo-Riemannian quadrics are preharmonic, and complete the classification of harmonic Killing fields in this case: up to pseudo-Riemannian congruence there is a unique harmonic Killing field on five of the six metrically distinct quadrics, the exception being the Riemannian $2$-sphere, on which no Killing field is harmonic (\cref{thm:harmkill}).  An interesting feature is the existence of a harmonic Killing field on the negative definite pseudo-hyperbolic plane, which is anti-isometrically dual to the Riemannian $2$-sphere, illustrating that although harmonic vector fields are invariant under isometry they are not invariant under anti-isometry.  However, further investigation (\cref{sec:PK}) shows that the combination of an anti-isometry with a para-K\"ahler twist does in fact preserve harmonic vector fields (\cref{prop:pktwist}).  When applied to the quadrics of neutral signature (viz.\ quotients of the de Sitter and anti-de Sitter planes), this yields a correspondence between harmonic Killing fields and harmonic conformal gradient fields, unifying results from Sections \ref{sec:HCGF} and \ref{sec:2d}.

\par
It may aid the reader to note that in most cases, when dealing with the differential $d\v$ of a smooth mapping $\v$ between manifolds, we omit the base point.  The exception is Section \ref{sec:2d}, where the base point is written as a subscript.

\par
The paper is based on parts of the first author's PhD thesis \cite{F}.  The authors express their thanks to the referee for a thorough reading of the manuscript.

\bigskip
%\pagebreak
\section{Generalised Cheeger-Gromoll metrics on pseudo-Riemannian vector bundles}\label{sec:GCGM}
             
A {\sl pseudo-Riemannian vector bundle\/} is a vector bundle $\pi\colon\E\to M$ equipped with a linear connection $\nabla$ and holonomy-invariant fibre metric $\<*,*\>$; thus:
$$
X\< \s,\tau\> = \<\nab X \s,\rho\> + \<\s,\nab X\rho\>,
$$
for all $X\in TM$ and all sections $\s,\rho\in\Gamma(\E)$, and $\<*,*\>$ is non-degenerate but not necessarily positive definite.  The motivating and most natural example is, of course, the tangent bundle of a pseudo-Riemannian manifold equipped with its Levi-Civita connection.  Let $K\colon T\E\to\E$ be the associated connection map, and let
$$
T\E=V\oplus H=\ker(d\pi)\oplus\ker(K)
$$
denote the splitting into vertical and horizontal distributions.  We also recall the following characteristic property of the connection map:    
\begin{equation}
K(d\s(X))=\nab X \s. 
\label{eq:cmap}
\end{equation}

\par 
Now let $g$ be a pseudo-Riemannian metric on $M$. The familiar construction of the Sasaki metric in the Riemannian case generalises naturally, yielding a pseudo-Riemannian metric $h$ on $\E$, which we continue to refer to as the Sasaki metric.  The construction of the generalised Cheeger-Gromoll metrics in the Riemannian case \cite{BLW1} may also be generalised, as follows.
Let $\E'\subset \E$ be the open dense subset: 
$$
\E'=\{e\in \E : \<e,e\> \neq -1\},
$$                          
and for each $(p,q)\in \R^2$ define a symmetric $(2,0)$-tensor $h_{p,q}$ on $\E'$ as follows:
\begin{align}
h_{p,q}(A,B) &= g(d\pi(A), d\pi(B)) 
\nonumber \\ 
&\qquad+\w^p(e)\bigl(\<K(A),K(B)\>+q\< K(A),e\>\< e,K(B)\>\bigr),
\label{eq:hpq}
\end{align}
for all $A,B\in T_e\E'$ and all $e\in\E'$, where $\w\colon\E'\to\R$ is the smooth function: 
$$
\w(e)={1}/{{\lvert 1+\< e,e\>\rvert}}.
$$  
If $q=0$ then $h_{p,q}$ is a pseudo-Riemannian metric on $\E'$ with the same signature as the Sasaki metric $h=h_{0,0}$.  However if $q\neq0$ then $h_{p,q}$ is of variable signature across $\E'$.  More precisely, if $q<0$ (resp.\ $q>0$) then $h_{p,q}$ has the same signature as the Sasaki metric in the region of $\E'$ where  $\< e,e\>< -1/q$ (resp.\ $\<e,e\>>-1/q$). Furthermore, for all $q\neq1$, $h_{p,q}$ degenerates mildly on the sphere bundle: 
$$
S\E(-1/q)=\{e\in\E: \< e,e\>=-1/q\},
$$
and if $q<0$ (resp.\ $q>0$) then the index of $h_{p,q}$ increases (resp.\ decreases) by $1$ in the space-like (resp.\ time-like) region where $\< e,e\>> -1/q$ (resp.\ $\<e,e\> < -1/q$).  Nevertheless, the parameters $(p,q)$ are referred to as the {\sl metric parameters\/} of the generalised Cheeger-Gromoll metric $h_{p,q}$.  If $p\leqs0$ then $h_{p,q}$ extends to $\E$, but degenerates drastically (to $\pi^*g$) on $S\E(-1)$ if $p<0$.  However if $p>0$ then $h_{p,q}$ becomes irremovably singular on $S\E(-1)$.

\bigskip  
\section{Harmonic sections}\label{sec:HS}

Let $\s$ be a section of $\E$, with {\sl pseudo-length\/} $\<\s,\s\>$ not identically $-1$; thus the preimage $\s\inv(\E')\subset M$ is a non-empty open subset.  The {\sl local $(p,q)$-energies\/} of $\s$ are defined:
$$
E_{p,q}(\s;U)=\int_U e_{p,q}(\s)\volume(g),
$$
for all relatively compact open subsets $U\subset \s\inv(\E')$, where $e_{p,q}(\s)\colon\s\inv(\E')\to\R$ is the {\sl $(p,q)$-energy density:}
$$
e_{p,q}(\sigma)=\tfrac{1}{2}h_{p,q}(d\s,d\s).
$$
Note that:
$$
h_{p,q}(d\s,d\s)=\sum_i \e_i\/h_{p,q}(d\s(E_i),d\sigma(E_i)),
$$      
for any $g$-orthonormal local tangent frame $\{E_i\}$ of $M$, where: 
\begin{equation}
\e_i=\< E_i,E_i\>=\pm1
\label{eq:indicator}
\end{equation} 
are the {\sl indicator symbols\/} of the frame.
It follows from \eqref{eq:cmap} and \eqref{eq:hpq} that:
\begin{align}
2e_{p,q}(\s)                          
&=n+\w^p(\s)\bigl(\<\nabla\s,\nabla\s\>
+q\/g(\nabla{F},\nabla{F})\bigr), 
\label{eq:hpqdensity}
\end{align} 
where $F=\frac12\<\s,\s\>$ and $\nabla F=\grad{F}$, the pseudo-Riemannian gradient vector field on $M$.  
%We refer to $\<\s,\s\>$ as the {\sl pseudo-length of $\s$.}  

\par
Composition of $d\s$ with the orthogonal projections of $T\E$ onto $V$ and $H$ yields the decomposition:
$$
d\sigma=d^v\sigma+d^h\sigma,
$$ 
and we define the {\sl vertical\/} and {\sl horizontal $(p,q)$-energy densities\/} by, respectively:
$$
e_{p,q}^v(\s)=\tfrac12 h_{p,q}(d^v\s,d^v\s),
\qquad 
e_{p,q}^h(\s)=\tfrac12 h_{p,q}(d^h\s,d^h\s),
$$
Since $V$ and $H$ are $h_{p,q}$-orthogonal distributions, the $(p,q)$-energy density splits:
$$
e_{p,q}(\s)=e_{p,q}^v(\s)+e_{p,q}^h(\s),
$$
and a brief further inspection of \eqref{eq:cmap} and \eqref{eq:hpq} reveals that:
$$
e_{p,q}^v(\s)=\tfrac12 \w^p(\s)(\<\nabla\s,\nabla\s\>
+qg(\nabla{F},\nabla{F})),
\qquad 
e_{p,q}^h(\s)=n/2.
$$
Thus the horizontal $(p,q)$-energy density is globally defined and constant, and
$$
E_{p,q}(\s;U)=E^v_{p,q}(\s;U)+\frac{n}{2}\Volume(U),
$$
where
\begin{equation*}
E_{p,q}^v(\s;U)=\int_U e_{p,q}^v(\s)\volume(g)
\end{equation*}
is the {\sl local vertical $(p,q)$-energy\/} of $\sigma$.

\medskip
\begin{defn}\label{defn:pqharm}
If the pseudo-length of $\s$ is not identically $-1$ then $\s$ is said to be a {\sl $(p,q)$-harmonic section\/} of $\E$ if:
$$
\frac{d}{dt}\Big\vert_{t=0}E^v_{p,q}(\s_t;U)=0,
$$
for all relatively compact open sets $U\subset \s\inv(\E')$ and all variations $\s_t$ of $\s$ through sections of $\E$ with $\s_t=\s$ on $M\ssm U$.  Note that $\s_t(U)\subset\E'$ for sufficiently small $t$.
\end{defn}     

\medskip                                      
The derivation of the Euler-Lagrange equations for this variational problem proceeds in a similar way to the Riemannian case \cite{BLW1}, but working in the pseudo-Riemannian environment requires additional technical vigilance.  Given a variation $\s_t$ as in \cref{defn:pqharm} the variation field $v_t$ is defined, as usual:
$$
v_t(x)=\frac{d}{dt}\s_t(x).
$$
Since $\s_t$ is a variation through sections, $v_t$ is a lift of $\s_t$ into $V$, which may be realised as a section $\rho_t$ of $\E$ by application of the connection map:
$$
\rho_t=K \circ v_t.
$$ 
Furthermore $\rho_t$ is compactly supported, within the closure $\bar U$.         
To simplify our main calculation it is convenient to split the first variation into two integrals as follows: 
\begin{align*}
\frac{d}{dt}\Big\vert_{t=0}E^v_{p,q}(\s_t;U)
&=\frac12 \int_U \frac{d}{dt}\Big\vert_{t=0}\w^p(\s_t)
\bigl(\<\nabla\s,\nabla\s\>+q\/g(\nabla F,\nabla F)\bigr)\volume(g) \\[1ex]
&\qquad
+\frac12 \int_U \w^p(\s)\frac{d}{dt}\Big\vert_{t=0}
\bigl(\<\nabla\s_t,\nabla\s_t\>+q\/g\bigl(\nabla F_t,\nabla F_t)\bigr)\volume(g) \\[1ex]
&=I_1+I_2.
\end{align*}        
We consider each integral in turn, introducing $\a=dF \otimes \s$, an $\E$-valued $1$-form on $M$, and abbreviating $\rho_0=\rho$. The proof of the following result (Lemma \ref{lem:pqharm}) is similar to that given in \cite{BLW1}; however note the appearance of an indicator symbol: 
$$
\e=\dfrac{1+2F}{\lvert 1+ 2F \rvert}=\pm1,
$$
to distinguish the cases $\<\s,\s\>>-1$ and $\<\s,\s\><-1$.
    
\medskip
\begin{lem}\label{lem:pqharm}\item{}

\medskip
\begin{enumerate}[leftmargin=1.7em, itemsep=1ex]

\item[\rm(1)]    
$\dfrac{d}{dt}\Big\vert_{t=0}\omega^p(\s_t)
=-2p\/\e\/\omega^{p+1}(\s) \<\s,\rho\>$.

\item[\rm(2)] 
$\dfrac{d}{dt}\Big\vert_{t=0} \<\nabla \s_t,\nabla \s_t\> 
=2\< \nabla \rho,\nabla \s\>$.

\item[\rm(3)] 
$\dfrac{d}{dt}\Big\vert_{t=0}\,g(\nabla F_t,\nabla F_t) 
=2\<\a,\nabla\rho\> + 2\<\nab{\nabla F}\s,\rho\>$.
\end{enumerate}
\end{lem}   
             
\medskip
\begin{prop}\label{prop:pqharm1}
The pieces of the first variation of the local vertical $(p,q)$-energy functional are:
\begin{align*} 
I_1&=-p\/\e\int_M \omega^{p+1}(\s)\<\s,\rho\> 
\bigl(\< \nabla\sigma,\nabla\sigma\> 
+qg( \nabla F,\nabla F) \bigr)\volume(g),  \\[1ex]
I_2&=\int_M \omega^p(\s)\bigl(\<\nabla\s+q\a,\nabla\rho\> 
+q\<\nab{\nabla F}\s,\rho\>\bigr)\volume(g).
\end{align*}
\end{prop}

\medskip
We now recall that if $\b$ is an $\E$-valued $1$-form on $M$, and $f\colon M\to\R$ a smooth function, then the following identity holds:
\begin{equation}
\nabla^*(f\b)= f\,\nabla^*\b-\b(\nabla{f}), 
\label{eq:codifffun}
\end{equation} 
where $\nabla^*\b=-\trace{\nabla\b}$, the pseudo-Riemannian codifferential.  

\medskip
\begin{lem}\label{lem:pqharm2}
The codifferential of $\g=\w^p(\s)(\nabla\s+q\a)$ is:
\begin{align*}
\nabla^*\g=\w^p(\s)\bigl(\nabla^*\nabla\s
+q((\Delta{F})\s-\nab{\nabla F}\s)\bigr)
+2p\/\e\,\w^{p+1}(\s)
\bigl(\nab{\nabla F}\s+q\/g(\nabla{F},\nabla{F})\s\bigr),
\end{align*}
where $\nabla^*\nabla=-\trace\nabla^2$ is the rough Laplacian, and $\Delta{F}= -\diverge{\nabla F}$ is the pseudo-Riemannian Laplace-Beltrami operator.
\end{lem}

\begin{proof}
Take $\b=\nabla\s+q\a$ and $f=\w^p(\s)$ in \eqref{eq:codifffun}.  Then:
$$
\nabla{f}=-2p\/\e\,\w^{p+1}(\s)\,\nabla{F},
$$
hence:
\begin{align*}
\nabla^*\g=\nabla^*(f\b) 
&=\w^p(\s)(\nabla^*\nabla\s+q\/\nabla^*\a)
+2p\/\e\,\w^{p+1}(\s)\bigl(\nab{\nabla F}\s
+qg(\nabla{F},\nabla{F})\s\bigr).
\end{align*}
Finally note that:
\[
\nabla^*\a=(\Delta{F})\s-\nab{\nabla{F}}\s. 
\qedhere
\]
\end{proof}

\medskip
We are now in a position to derive the Euler-Lagrange equations for $(p,q)$-harmonic sections.

\medskip
\begin{thm}\label{thm:pqharm}
Let $\s$ be a section of pseudo-Riemannian vector bundle $\E\to M$ over a pseudo-Riemannian manifold, with pseudo-length not identically $-1$.  Then $\s$ is a $(p,q)$-harmonic section if and only if $\tau_{p,q}(\s)=0$, where $\tau_{p,q}(\s)$ is the following {\rm Euler-Lagrange section} of $\E$:
$$
\tau_{p,q}(\s)=T_p(\s)-\phi_{p,q}(\s)\s,
$$
with $T_p(\s)\in\Gamma(\E)$ and $\phi_{p,q}(\s)\colon M\to\R$ defined:
\begin{align*}
T_p(\s)
&=(1+2F)\nabla^*\nabla\s+2p\,\nab{\nabla{F}}\s,  \\[1ex]
\phi_{p,q}(\s)
&=p\<\nabla{\s},\nabla{\s}\>-pq\,g(\nabla{F},\nabla{F})-q(1+2F)\Delta{F}.
\end{align*}
\end{thm}

\begin{proof}
By \cref{prop:pqharm1}:
\begin{align*}
I_1&=-p\,\int_M \e\/\w^{p+1}(\s) 
\bigl\< \bigl(\< \nabla\s,\nabla\s\>+qg(\nabla F,\nabla F)\bigr)\s,\rho\bigr\>\volume(g),    \\
I_2&=\int_M \bigl\<\nabla^*\g+q\/\w^p(\s)\nab{\nabla F}\s,\rho\bigr\>\volume(g),\end{align*}
where we have used integration by parts to rewrite $V_2$ in divergence form.  Now by \cref{lem:pqharm2}, after a cancellation of terms:
\begin{align*}
I_2=\int_M \e\/\w^{p+1}(\s)
\bigl\<\e\lvert 1&+2F\rvert
(\nabla^*\nabla \s+q(\Delta{F})\s) \\
&+2p(\nab{\nabla{F}}\s 
+q\/g(\nabla{F},\nabla{F})\s),\rho\bigr\>\volume(g).
\end{align*}
Therefore:
$$
I_1+I_2=\int_M \e\,\w^{p+1}(\s)\<\tau_{p,q}(\s),\rho\>\volume(g),
$$
noting that: 
$$
\e\lvert 1+2F\rvert=1+2F.
$$
The result now follows from $L^2_{\text{loc}}$-non-degeneracy: if $\xi$ is a section of a pseudo-Riemannian vector bundle $\E\to M$, and 
$$
\int_U \<\xi,\rho\>\volume(g)=0
$$ 
for all relatively compact open $U\subset M$ and all $\rho\in\Gamma(\E)$ with support in $\bar U$, then $\xi=0$.
\end{proof}

\medskip  
\begin{rems}\label{rems:pqharm}\item{}
\begin{enumerate}[leftmargin=1.8em]

\item 
The Euler-Lagrange equations resolve the singularity in the vertical $(p,q)$-energy functional: they are valid on all of $M$, not just on $\s\inv\E'$.

\item 
If $\<\s,\s\>\equiv k\neq-1$ then the Euler-Lagrange equations reduce to:
$$
(1+k)\nabla^*\nabla\s=p\<\nabla\s,\nabla\s\>\s.
$$
If $k\neq0$ and $p=1+1/k$ then this is the equation for $\s$ to be a harmonic section of the sphere bundle $S\E(k)$ equipped with the restriction of the Sasaki metric.  Thus, for all $k\neq -1,0$, harmonic sections of $S\E(k)\to M$ are precisely the $(p,q)$-harmonic sections of $\E$ of constant pseudo-length $k$, for $p=1+1/k$ and all $q\in\R$.

\item 
If $\<\s,\s\>\equiv-1$ (ie.\ $\s\inv\E'=\emptyset$) then $T_p(\s)\equiv0$ and $\phi_{p,q}(\s)=p\<\nabla\s,\nabla\s\>$.  We therefore extend the terminology and decree that $\s$ is $(0,q)$-harmonic for all $q\in\R$.

\item 
If $\s$ is parallel then $\s$ is $(p,q)$-harmonic for all $(p,q)$.\end{enumerate}
\end{rems}

\medskip
The following definition generalises that of \cite{BLW2}.

\medskip
\begin{defn}\label{defn:preharm}
A section $\s$ of a pseudo-Riemannian vector bundle over a pseudo-Riemannian manifold is said to be {\sl $p$-preharmonic\/} if $T_p(\s)$ is pointwise collinear with $\s$, and {\sl preharmonic\/} if $\s$ is $p$-preharmonic for all $p$.\end{defn}

\medskip
Preharmonicity means:
\begin{enumerate}[leftmargin=1.5em]
\item[i)] 
There exists a smooth function $\nu\colon M\to\R$ such that $\nabla^*\nabla\s=\nu\s$; for example, if $\s$ is an eigenfunction of the rough Laplacian.
\item[ii)] 
There exists a smooth function $\zeta\colon M\to\R$ such that $\nab{\nabla F}\s=\z\/\s$.
\end{enumerate}
As in \cite{BLW2}, we refer to $\z$ as the {\sl spinnaker\/} of $\s$.  The following result is a direct generalisation of the Riemannian version used in \cite{BLW2}.

\medskip
\begin{thm}\label{prop:preharm}
Let $\s$ be a preharmonic section of a pseudo-Riemannian vector bundle over a pseudo-Riemannian manifold.  Then $\s$ is a $(p,q)$-harmonic section if and only if:
$$
(p+q+2qF)\/\Delta{F}+2p(1+qF)\/\z+(1+2(1-p)F)\/\nu=0.
$$
\end{thm}

\begin{proof}
This follows from \cref{thm:pqharm} and the Weitzenb{\"o}ck identity:\begin{equation}
\< \nabla^*\nabla\s,\s\>=\<\nabla\s,\nabla\s\>+\Delta{F}, 
\label{eq:weitz}
\end{equation}
which continues to hold in the pseudo-Riemannian case.
\end{proof}

\bigskip
\section{Harmonic vector fields and pseudo-Riemannian hyperquadrics}
\label{sec:HVF}

Henceforward we specialise to the case $\E=TM$ for a pseudo-Riemannian manifold $M$, with $\nabla$ the Levi-Civita connection and $\<*,*\>=g$, the pseudo-metric on $M$.  In this case, sections of $\E$ are of course vector fields on $M$.

\medskip
\begin{defn}\label{defn:hvf}
A vector field $\s$ on $(M,g)$ is said to be a {\sl harmonic vector field\/} if $\s$ is a $(p,q)$-harmonic section of the tangent bundle $TM$ for some $(p,q)$; otherwise said, the Euler-Lagrange vector field $\tau_{p,q}(\s)=0$ identically, by Theorem \ref{thm:pqharm}.  The pair $(p,q)$ are said to be {\sl metric parameters\/} for the harmonic vector field $\s$.
\end{defn}

\medskip
The metric parameters for a harmonic vector field need not be unique, even for vector fields of non-constant length.  This was observed in the Riemannian case \cite{BLW2}, and we will exhibit further non-Riemannian examples in \cref{thm:cgf}.
          
The natural action of the isometry group of $(M,g)$ on vector fields is via the push-forward construction:
$$
(\v.\s)(x)
=(\v_*\s)(x)
=d\varphi(\s(\v\inv(x))),
$$
for all isometries $\v$ and all $x\in M$. The vector fields $\s$ and $\v.\s$ are then said to be {\sl congruent.}  As in the Riemannian case, harmonic vector fields are determined up to congruence:

\medskip
\begin{thm}\label{thm:hvfcong}
Let $\s$ be a harmonic vector field on a pseudo-Riemannian manifold $(M,g)$, and let $\v$ be an isometry of $(M,g)$.  Then $\varphi.\sigma$ is also harmonic, with the same metric parameters.
\end{thm}

\begin{proof}
Pseudo-Riemannian isometries are totally geodesic: $\nabla d\varphi=0$.  It then follows from \cref{thm:pqharm} that:
\[
\tau_{p,q}(\v.\s)=\v.\tau_{p,q}(\s).
\qedhere
\]
\end{proof}

\medskip
\begin{rem}\label{rem:hvf}
Although harmonic vector fields are invariant under isometries, in general (perhaps surprisingly) they are not invariant under homotheties; examples of this were already noted in \cite{BLW1}.  Consequently, when solving the Euler-Lagrange equations scale factors play a non-trivial r\^ole.
\end{rem}

\medskip
Recall that a {\sl space form\/} is a simply-connected complete pseudo-Riemannian manifold of constant sectional curvature, and two space forms are isometric if and only if they have the same dimension, index and sectional curvature \cite[Proposition 8.23]{ON}.  For computational and geometric purposes we work with hyperquadric models, which in some cases are only locally isometric to the corresponding space form.  Let $\R^{n+1}_u$ denote pseudo-Euclidean space of index $u\in\{0,\dots,n+1\}$, with inner product:
\begin{equation}
\<x,y\>= x_1y_1+\cdots+x_{n+1-u}\/y_{n+1-u}-\cdots-x_{n+1}\/y_{n+1},
\label{eq:pseudoeuc}
\end{equation}
and let $Q\colon \R^{n+1}_u\to\R$ be the associated quadratic form: $Q(x)=\<x,x\>$.

\medskip
\begin{defn}\label{defn:pseudo}
The {\sl pseudo-sphere\/} (resp.\ {\sl pseudo-hyperbolic space\/}) of dimension $n\geqs2$, index $v\in\{0,\dots,n\}$ and radius $r>0$ is the hyperquadric:
\begin{gather*}
S^n_v(r)=\{x\in\R^{n+1}_v:Q(x)=r^2\} \\
\text{(resp.\ $H^n_v(r)=\{x\in\R^{n+1}_{v+1}:Q(x)=-r^2\}$)}
\end{gather*}
equipped with the induced metric.  The sectional curvature is $1/r^2$ (resp.\ $-1/r^2$).
\end{defn}

\medskip\noindent
All the hyperquadrics are connected, except the extreme cases $S^n_n(r)$ and $H^n_0(r)$ which have precisely two connected components \cite[Ch.\ 4, Lemma 25]{ON}, one of which is normally discarded.  The pseudo-spheres and pseudo-hyperbolic spaces of unit radius are abbreviated $S^n_v$ and $H^n_v$, respectively.

We recall also that a diffeomorphism $\v\colon(M,g)\to(N,h)$ of pseudo-Riemannian manifolds is an {\sl anti-isometry\/} if $\v^*h=-g$.
%$$
%h(d\v(X),d\v(Y))=-g(X,Y), \quad\text{for all $X,Y\in\Gamma(TM)$}.
%$$
Note that for two pseudo-Riemannian $n$-manifolds to be anti-isometric the sum of their indices must equal $n$.  The pseudo-Euclidean anti-isometry:
\begin{equation}
\v\colon \R^{n+1}_{v+1}\to\R^{n+1}_{n-v};\; 
\v(x_1,\ldots,x_{n+1})=(x_{n+1-v},\ldots,x_{n+1}, x_1, \ldots, x_{n-v})
\label{eq:antisom}
\end{equation}
carries $H^n_v(r)$ anti-isometrically onto $S^n_{n-v}(r)$; its restriction is the {\sl canonical anti-isometry\/} between these two hyperquadrics \cite[Ch.\ 4, Lemma 24]{ON}.  (The slight difference with \cite{ON} arises from our definition \eqref{eq:pseudoeuc} of the pseudo-Euclidean inner product.)  In pseudo-Riemannian geometry anti-isometric spaces are often considered to be identical.  However, although anti-isometries are totally geodesic, from the viewpoint of harmonic vector fields they are not so natural, essentially because the term $1+2F$ in the Euler-Lagrange vector field $\tau_{p,q}(\s)$ (see \cref{thm:pqharm}) is not invariant.  Thus if $\s$ is a harmonic vector field on $(M,g)$ and $\v\colon (M,g)\to(N,h)$ is an anti-isometry then the push-forward $\v_*\s$ need not be a harmonic vector field on $(N,h)$; a concrete example is given in \cref{sec:HCGF} (see \cref{eg:cgf}).

\bigskip
\section{Harmonic conformal gradient fields}\label{sec:HCGF}
    
The construction of conformal gradient fields on Riemannian space forms generalises to pseudo-Riemannian hyperquadrics.  Let $M=S^n_v$ or $M=H^n_v$, and let $\mathbb{V}$ denote the appropriate ambient pseudo-Euclidean space (see Definition \ref{defn:pseudo}).  Note that the equation of the hyperquadric is $\<x,x\>=\e$ where $\e=\pm 1$ is the sectional curvature. Let $a\in \mathbb{V}$ have pseudo-length
$$
\mu=\<a,a\>,
$$ 
and let $\a\colon M\to \R$ be the restriction to $M$ of the covector metrically dual to $a$:
$$
\a(x)=\<x,a\>,
$$
for all $x\in M$.  The {\sl conformal gradient field\/} $\s$ on $M$ with {\sl pole vector\/} $a$ is then defined:
$$
\s=\grad{\a}=\nabla\a,
$$
where the gradient is, of course, that intrinsic to the hyperquadric.  We now record some relevant properties of pseudo-Riemannian conformal gradient fields, computations of which are essentially identical to those given in \cite[Section 3]{BLW2}.

\medskip
\begin{prop}\label{prop:cgf}
Let $\s$ be a conformal gradient field on $M$, with pole vector $a$. Then for all $x\in M$ and $X,Y\in T_xM$:
\begin{enumerate}[leftmargin=1.75em, itemsep=0.5ex]
\item[\rm{(1)}] 
$\s(x)=a-\e\/\a(x)x.$
\item[\rm{(2)}] 
$2F=\<\s,\s\>=\mu-\e\a^2.$
\item[\rm{(3)}] 
$\nab X\s=-\e\/\a\/ X.$
\item[\rm{(4)}] 
$\nabsq{X}{Y}\s=-\e \<\s,X\> Y.$
\end{enumerate}
\end{prop}

\medskip
%\pagebreak
\begin{rems}\label{rems:cgf}
\item{}
\begin{enumerate}[leftmargin=1.8em]

\item 
By \cref{prop:cgf}\,(1), if $\v\colon H^n_v\to S^n_{n-v}$ is the canonical anti-isometry, and $\s$ is a conformal gradient field on $H^n_v$ with pole vector $a$, then $\v_*\s$ is a conformal gradient field on $S^n_{n-v}$ with pole vector $\v(a)$.

\item
It follows from \cref{prop:cgf}\,(2) that $\s(x)$ is a null vector if and only if $\e\/\mu>0$ and $x=\pm a/\sqrt{\lvert\mu\rvert}$. But then $\s(x)=0$ by (1).  Therefore $\s$ is either space-like or time-like, although it is not possible to discern which from the signs of $\mu$ and $\e$. If $\e\/\mu<0$ then $\s$ has no zeros.
\end{enumerate}
\end{rems}

\medskip
\begin{prop} \label{prop:cgfpreharm}
If $\s$ is a conformal gradient field then $\s$ is preharmonic, with 
$\nu=\e$ and spinnaker $\z= \e(\mu-2F)$.
\end{prop}

\begin{proof}
We calculate:
\begin{align*}
\nabla^*\nabla\s
&=-\trace{\nabla^2\s}
=-\ts\sum_i\e_i\,\nabla^2_{E_i,E_i}\s \\
&=\ts\sum_i\e_i\,\e\<\s,E_i\> E_i,
\quad\text{by \cref{prop:cgf}\,(4)} \\
&=\e\/\s,
\end{align*}
hence $\nu=\e$.  Furthermore:
\begin{align}
\nabla{F}
&=-\e\/\a \ts\sum_i\e_i\<\s,E_i\> E_i
=-\e\/\a\/\s. 
\label{eq:cgfgrad}
\end{align}
Therefore by \cref{prop:cgf}\,(3):
\begin{align*}
\nab{\nabla{F}}\s
&=-\e\/\a\,\nabla{F}=\a^2\s=\e(\mu-2F)\s,
\end{align*}
hence $\z=\e(\mu-2F)$.
\end{proof}

\medskip
\begin{thm}\label{thm:cgf}
Let $\s$ be a conformal gradient field on a pseudo-Riemannian hyperquadric, whose pole vector has pseudo-length $\mu\in\R$.  
\begin{itemize}[leftmargin=1.7em]
\item[\rm(1)]
If $\mu\geqs 0$ then $\s$ is $(p,q)$-harmonic if and only if:
$$
n>2, \quad \mu =1/(n-2), \quad p=n+1, \quad q=2-n.
$$
\item[\rm(2)]
If $\mu<0$ then $\s$ is $(p,q)$-harmonic if and only if $\mu=-1$ and either:
$$
p=n+1, \quad q=\frac{1+n-n^2}{n},
$$
or:
$$
n>2, \quad p=1/(2-n), \quad q=0.
$$
\end{itemize}
\end{thm}

\begin{proof}
Since $\s$ is preharmonic the harmonic equations simplify to those of \cref{prop:preharm} with $\nu=\e$ and $\z=\e(\mu-2F)$.  By \cref{prop:cgf} and \eqref{eq:cgfgrad} the Laplacian of $F$ is:
\begin{align*}
\Delta{F}
&=-\diverge{\nabla{F}}
=\e\<\s,\s\>-n\a^2=2\e\/F(1+n)-\e\/n\mu.
\end{align*}
Therefore the harmonic equations reduce to the following polynomial in $F$:\begin{align*}
0&=(p+q+2qF)(2(1+n)F-n\mu)+2p(1+qF)(\mu-2F)+1+2(1-p)F.
\end{align*}
This is in fact the same polynomial that appears in the Riemannian case \cite[Theorem 3.2]{BLW2}, and the analysis proceeds in the same way.
\end{proof}

\medskip
It is interesting to note that \cref{thm:cgf} does not depend on the curvature of the hyperquadric. However it does depend on the index of the ambient space: if this is strictly positive (resp.\ negative) definite then necessarily $\mu>0$ (resp.\ $\mu<0$).  In particular, this precludes the existence of harmonic conformal gradient fields on the Riemannian $2$-sphere.  It should also be noted that although harmonic conformal gradient fields are metrically unique if $\mu>0$, if $\mu<0$ and $n>2$ there are two sets of metric parameters.  However if $n=2$ the metric parameters are unique, and equal to $(3,-1/2)$ for all quadrics (other than the Riemannian $2$-sphere).

\par
Finally we note that harmonic conformal gradient fields are uniquely determined up to congruence by the pseudo-length of the pole vector:      

\medskip
\begin{thm}\label{thm:cgfcong}
The congruence class of a conformal gradient field on a pseudo-Riemannian hyperquadric is determined by the pseudo-length of its pole vector.
\end{thm}

\begin{proof}
Let $\s, \tilde\s$ be conformal gradient fields with pole vectors $a,\tilde{a}$ respectively, such that $\mu=\tilde{\mu}$.  There exists an ambient isometry $\Phi\in O^{++}(n+1,u)$, where $u$ is the index of $\mathbb{V}$, such that $\Phi(a)=\tilde{a}$.  The potential $\tilde\a$ is:
$$
\tilde\a(x)=\<\tilde{a},x\>=\<\Phi(a),x\>=\<a,\Phi\inv(x)\>;
$$
thus:
$$
\tilde\a=\a\circ\Phi\inv.
$$
For all $X\in T_x M$:
\begin{align*}
\<\nabla\tilde\a,X\> 
&=d\tilde\a(X)=d\a(d\Phi\inv(X))
=\<\nabla\a, d\Phi\inv(X)\> \\
&=\<\nabla\a,\Phi\inv(X)\>
=\<\Phi(\nabla\a), X\>
=\<d\Phi(\nabla\a),X\>,
\end{align*}
where $\nabla\a$ is evaluated at $\Phi\inv(x)$.  Therefore:
$$
\tilde\s(x)=\nabla\tilde\a(x)
=d\Phi(\nabla\a(\Phi\inv(x)))
=d\Phi\circ\s\circ \Phi\inv(x).
$$
Hence $\tilde\s=\v.\s$ where $\v=\Phi\vert_M$.
\end{proof} 

\medskip        
\begin{eg}\label{eg:cgf}            
Consider $M=H^2_2$, whose underlying manifold is the standard $2$-sphere $x^2+y^2+z^2=1$.  By Theorem \ref{thm:cgf}\,(2) the conformal gradient field with pole vector $(0,0,1)$ is $(3,-1/2)$-harmonic. This vector field has two zeros, at $\pm(0,0,1)$, and up to congruence is the unique harmonic conformal gradient field on $M$.  In contrast, by Theorem \ref{thm:cgf}\,(1) the Riemannian $2$-sphere $S^2_0$ has no harmonic conformal gradient fields.  Furthermore $H^2_2$ and $S^2_0$ are anti-isometric, the canonical anti-isometry \eqref{eq:antisom} being the identity map, and the push-forward of $\s$ to $S^2_0$ is also a conformal gradient field (\cref{rems:cgf}), illustrating that harmonic vector fields are not invariant under anti-isometry.
\end{eg}

\bigskip
\section{Preharmonic Killing fields on pseudo-Riemannian hyperquadrics}\label{sec:HKF}
        
Now let $\s$ be a Killing field on a pseudo-Riemannian hyperquadric $M$ of sectional curvature $\e=\pm 1$. Then $\s$ is the restriction to $M$ of a unique skew-symmetric linear transformation $A\colon \mathbb{V}\to\mathbb{V}$, which we refer to as the {\sl linear extension\/} of $\s$.  Thus if $A$ has matrix $(a_{ij})$ with respect to an orthonormal frame of $\mathbb{V}$ then:
\begin{equation}
a_{ij}=-\e_i\/\e_j\/a_{ji},
\label{eq:matrix}
\end{equation}     
where the $\e_i$ are the indicator symbols of the frame.  It follows from the pseudo-Riemannian Gauss formula \cite{ON} that for all $X\in T_xM$ and all $x\in M$: 
\begin{equation}
\nab X \s=A(X)-\e\<A(X),x\>x, 
\label{eq:kfcov}
\end{equation}
where $x$ is regarded as a unit normal field on $M$.  Note that since $A$ is skew-symmetric so is $A^3$, which is therefore the linear extension of a Killing field $\hat\s$ on $M$.

\medskip
\begin{lem}\label{lem:kf1}
If $\s$ is a Killing field  on a pseudo-Riemannian hyperquadric $M$ of curvature $\e$ then:
$$
\nab{\nabla{F}}\s=-\hat\s-2\e\/ F\s.
$$
\end{lem}   

\begin{proof}
We note first that an orthonormal tangent frame $\{E_i\}$ to $M$ at $x\in M$, with indicator symbols $\e_i$, extends to an orthonormal basis $\{E_1,\dots,E_n,x\}$ of $\mathbb{V}$, with indicator symbols $\e_1,\dots,\e_n,\e$.  Then, since $2F=\<\s,\s\>$ we have:
\begin{align*}
\nabla{F}(x)    
&=\ts\sum_i \e_i\/dF(E_i)E_i
=\ts\sum_i \e_i \<\nab{E_i}\s,\s\>E_i \\
&=\ts\sum_i \e_i \< A(E_i),A(x)\>E_i, 
\quad\text{by \eqref{eq:kfcov}} \\
&=-\ts\sum_i \e_i \< E_i,A^2(x)\>E_i 
=-A^2(x)+\e\/\<A^2(x),x\>x \\
&=-A^2(x)-\e\<\s(x),\s(x)\>x
=-A^2(x)-2\e\/F(x)x.
\end{align*}
Therefore by \eqref{eq:kfcov} again:
\begin{align*}
\nab{\nabla{F}(x)}\s   
&=A(\nabla{F}(x))-\e\<A(\nabla{F}(x)),x\> x \\
&=-A^3(x)-2\e\/F(x)\/A(x) \\
&=-\hat\s(x)-2\e\/F(x)\s(x),
\end{align*}
since $\< A^3(x),x\>=0=\<A(x),x\>$.
\end{proof}

\medskip
In order to determine which Killing fields are preharmonic we will use the following technical fact.

\medskip
\begin{lem}\label{lem:kfpwise}
Suppose $\s,\rho$ are non-trivial Killing fields on a pseudo-Riemannian hyperquadric $M$.  If $\rho=\lambda\/\s$ for some smooth function $\lambda\colon M\to\R$ then $\lambda$ is constant.
\end{lem}

\begin{proof}
%For convenience we assume $M$ has unit radius.
Suppose $\s,\rho$ have skew-symmetric linear extensions $A,B$ respectively.  
Then for all $x\in M$: 
\begin{equation}
B(x)-\lambda(x)A(x)=0.
\label{eq:pwise}
\end{equation}
Differentiating this equation and rearranging yields:
\begin{equation}
B(X)-\lambda(x)A(X)=d\lambda(X)A(x), 
\label{eq:pwisediff}
\end{equation}
for all $X\in T_x M$.  Since $x$ is normal to $M$, it follows from \eqref{eq:pwise} and \eqref{eq:pwisediff} that for each $x\in M$ the skew-symmetric linear map $B-\lambda(x)A\colon\mathbb{V}\to\mathbb{V}$ has rank at most one.  However (non-trivial) skew-symmetric transformations of pseudo-Euclidean space have rank at least two by \eqref{eq:matrix}.  Therefore $B-\lambda(x)A=0$, and consequently $\lambda(x)=\lambda(y)$ for all $x,y\in M$.  
%By rank/nullity, $\ker(A)$ is a subspace of codimension at least $2$, whose intersection with $M$ is therefore a (totally geodesic) submanifold of dimension at most $n-2$.  Therefore $d\lambda=0$ on a connected open dense subset, hence by continuity $d\lambda$ vanishes everywhere; so $\lambda$ is constant. 
\end{proof}

\medskip
\begin{prop}\label{prop:kfpreharm}
A Killing field $\s$ on a pseudo-Riemannian hyperquadric of curvature $\epsilon$ is preharmonic if and only if $\hat\s=\lambda\/\s$ for some $\lambda\in\R$, in which case the spinnaker is:
$$
\z=-(\lambda+2\e F).
$$
\end{prop}
    
\begin{proof}
Since $\s$ is a Killing field we have \cite{YB}:
\begin{equation}
\nabla^*\nabla\s=\Ric(\s)=\e(n-1)\s;
\label{eq:kfrough}
\end{equation}
thus $\s$ is an eigenfunction of the rough Laplacian.
It therefore follows from Lemma \ref{lem:kf1} that $\s$ is preharmonic (\cref{defn:preharm} et seq.)~if and only if $\hat \s$ is a pointwise scalar multiple of $\s$, and thus from Lemma \ref{lem:kfpwise} that $\hat\s=\lambda\/\s$ for some $\lambda\in\R$.  The spinnaker may be read off from \cref{lem:kf1}.
\end{proof}

\medskip
We also require the Laplacian of the pseudo-length of a Killing field.

\medskip
\begin{lem}\label{lem:kf2}
The pseudo-length of a Killing field on a hyperquadric of curvature $\e$ satisfies: 
$$
\Delta{F}=2\e(n+1)F-\<A,A\>,
$$
where $\<A,A\>$ is the pseudo-length of the linear extension.
\end{lem}

\begin{note*}
The pseudo-length of $A$ is measured with respect to the metric on $\mathbb{V}^*\otimes\mathbb{V}$ inherited from the metric \eqref{eq:pseudoeuc} on $\mathbb{V}$:
$$
\<A,A\>=\ts\sum_i \e_i\<A(e_i),A(e_i)\>,
$$
where $\{e_i\}$ is any orthonormal basis of $\mathbb{V}$, with indicator symbols $\e_i$.
\end{note*}

\begin{proof}
Firstly, from the Weitzenb\"ock formula \eqref{eq:weitz} and \eqref{eq:kfrough}:
\begin{align*}
\Delta{F}
&=\<\nabla^*\nabla\s,\s\>-\<\nabla\s,\nabla\s\>
=2\e(n-1)F-\<\nabla\s,\nabla\s\>.
\end{align*}
Now, recalling the note at the beginning of the proof of Lemma \ref{lem:kf1}, by \eqref{eq:kfcov}:
\begin{align*}
\<\nabla\s,\nabla\s\> 
&=\ts\sum_i \e_i\<\nab{E_i}\s,\nab{E_i}\s\> \\
&=\ts\sum_i \e_i \bigl(\<A(E_i),A(E_i)\>-\e\<A(E_i),x\>^2\bigr) \\
&=\<A,A\>-\e\< A(x),A(x)\>-\e \ts\sum_i\e_i\<\s(x),E_i\>^2 \\
&=\<A,A\>-2\e\<\s,\s\>
=\<A,A\>-\e\/F.
\qedhere
\end{align*}     
\end{proof}

\medskip       
Combining \cref{prop:kfpreharm} and \cref{lem:kf2} with \cref{prop:preharm} yields the following criterion for a preharmonic Killing field to be harmonic.

\medskip
\begin{thm}\label{thm:kf}
Let $\s$ be a preharmonic Killing field on a pseudo-Riem\-annian hyperquadric of curvature $\e$.  Then $\s$ is $(p,q)$-harmonic if and only if:
\begin{align*}
0 &=\e(n+1-p)q\/(2F)^2 \\
&\qquad
+\bigl(\e(n-1+(n+1)q)-pq\lambda-q\<A,A\>\bigr)(2F) \\
&\qquad\qquad
+\e(n-1)-2p\lambda-(p+q)\<A,A\>,
\end{align*}
where $A$ is the linear extension of $\s$ and $\lambda\in\R$ is  characterised by $2F\/\lambda=\<\s,\hat \s\>$.
\end{thm}

\medskip        
We will see that in the $2$-dimensional case all Killing fields are preharmonic.

\bigskip
\section{Harmonic Killing fields on pseudo-Riemannian quadrics}\label{sec:2d}

In this section we work in pseudo-Euclidean $3$-space, where for convenience the coordinates are denoted $(x,y,z)$ rather than $(x_1,x_2,x_3)$.  
We recall that there are six pseudo-Riemannian quadrics, oganised into three anti-isometric pairs:

\begin{itemize}[leftmargin=1em]

\item 
The Riemannian $2$-sphere $S^2_0\subset\R^3_0$ and its negative definite counterpart $H^2_2\subset\R^3_3$, whose underlying manifold is the standard $2$-sphere $x^2+y^2+z^2=1$.

\item The hyperbolic plane $H^2_0\subset\R^3_1$ and its negative definite counterpart $S^2_2\subset\R^3_2$, whose underlying manifolds are the hyperboloids of two sheets with equations $x^2+y^2-z^2=-1$ and $x^2-y^2-z^2=1$, respectively.  (Strictly speaking, the hyperbolic plane is a connected component of $H^2_0$.)

\item The neutral quadrics, $S^2_1\subset\R^3_1$ and $H^2_1\subset\R^3_2$, whose underlying manifolds are the hyperboloids of one sheet with equations $x^2+y^2-z^2=1$ and $x^2-y^2-z^2=-1$, respectively.
\end{itemize}
Note that the quadrics of index $0$ and $2$ are in fact space forms, whereas strictly speaking the neutral quadrics are not.

\medskip
\begin{lem}\label{lem:2dkfpreharm}
Let $\s$ be a Killing field on a pseudo-Riemannian quadric of curvature $\e$, whose linear extension has the following matrix with respect to an orthonormal frame of $\mathbb{V}$:
$$
\left(\begin{array}{ccc}
0 & a & b \\
-\e_1\/\e_2\/a & 0 & c \\
-\e_1\/\e_3\/b & -\e_2\/\e_3\/c & 0 
\end{array} \right)
$$
where $a,b,c\in\R$ and $\e_1,\e_2,\e_3$ are the indicator symbols of the frame.  Then $\s$ is preharmonic, and:
$$
\lambda=-\e_1\/\e_2\/a^2-\e_1\/\e_3\/b^2-\e_2\/\e_3\/c^2.
$$
\end{lem}

\begin{proof}
By \cref{prop:kfpreharm} it suffices to calculate $A^3$ and compare it with $A$.
\end{proof}

\medskip
\begin{thm}\label{thm:2dkf}
Let $\sigma$ be a Killing field on a pseudo-Riemannian quadric of sectional curvature $\e$. Then $\s$ is $(p,q)$-harmonic if and only if: 
$$
p=3,\quad q=-1/2,\quad \lambda=\e.
$$
\end{thm}

\begin{proof}
Consider first:
\begin{equation*}
\<A,A\>=\textstyle\sum_i\e_i\/\< A(e_i),A(e_i)\>
=\textstyle\sum_{i,j}\e_i\/\e_j\,a_{ij}^{\;2} 
=-2\lambda,
\end{equation*}
by \cref{lem:2dkfpreharm}.  Therefore, since $\s$ is preharmonic, by \cref{thm:kf} $\s$ is $(p,q)$-harmonic if and only if:
\begin{align*}
0&=\e(3-p)q(2F)^2
+\bigl(\e(1+3q)+(2-p)q\lambda\bigr)(2F)
+2q\lambda+\e.
\end{align*}
The leading coefficient of this polynomial in $F$ vanishes if and only if $p=3$ or $q=0$; however if $q=0$ the linear term cannot vanish.  When $p=3$ the remaining equations reduce to:
$$
\e(1+3q)-q\lambda=0=\e+2q\lambda,
$$
which yield the stated values of $q$ and $\lambda$.
\end{proof}                     

\medskip
We note that for the Riemannian $2$-sphere $\lambda<0$ and $\e=1$, so \cref{thm:2dkf} precludes the existence of harmonic Killing fields, as already observed in \cite{BLW2}.  Comparison of \cref{lem:2dkfpreharm} and \cref{thm:2dkf} shows that harmonic Killing fields on each of the remaining pseudo-Riemannian quadrics form a quadric in the $3$-dimensional Lie algebra of Killing fields (although not necessarily of the same type as the underlying quadric or its anti-isometric counterpart).  However we will show that this quadric is actually a single congruence class.
%thus up to congruence there is a unique harmonic Killing field.  
In fact we will show that the congruence class of a Killing field on a pseudo-Riemannian quadric is determined by $\lambda$. This was already observed for $S^2_0$ and $H^2_0$ in \cite{BLW2}, from which it may be deduced also for $H^2_2$ and $S^2_2$, since the space of of Killing fields and its congruence structure is preserved by the canonical anti-isometry, leaving only the neutral quadrics.  It suffices to consider $H^2_1$, and we first establish the qualitative behaviour of Killing fields in this case.

\medskip
%\pagebreak
\begin{prop}\label{prop:2dkffix}
The fixed points of a non-trivial Killing field on $H^2_1$ are categorised by $\lambda=a^2+b^2-c^2$.  The Killing field has:

\begin{enumerate}[leftmargin=1.7em]
\item[\rm{(1)}] 
no fixed points if $\lambda<0;$
\item[\rm{(2)}] 
two ideal fixed points, one on each component of the boundary at infinity, if $\lambda = 0;$
\item[\rm{(3)}] 
two fixed points if $\lambda>0$.
\end{enumerate}
\end{prop}

\begin{proof}
The idea is to set up a finite model for $H^2_1$, analogous to the Beltrami disc model for the hyperbolic plane.
Let $C\subset \R^3_2$ be the finite open cylinder:
$$
C=\{(x,y,z): -1<x<1,\; y^2+z^2=1\},
$$
and project $H^2_1$ onto $C$ along rays through the origin. This gives a map:
$$
\psi \colon H^2_1 \to C;\; \psi(x,y,z)=\frac{1}{\sqrt{1+x^2}}(x,y,z),
$$
with differential:
\begin{align*}
d\psi_{(x,y,z)}(u,v,w) 
&=\frac{1}{(1+x^2)^{3/2}}
\bigl(u,-xyu+(1+x^2)v,-xzu+(1+x^2)w\bigr).
\end{align*}
The inverse map is:
\begin{align*}
\psi\inv(\bar{x},\bar{y},\bar{z}) 
&=\frac{1}{\sqrt{1-\bar{x}^2}}({\bar{x}},{\bar{y}},{\bar{z}})=(x,y,z).\end{align*}
The components of $\s(x,y,z)$ are:
$$
u= ay+bz,\quad v= ax+cz,\quad w=bx-cy.
$$               
Therefore the projection $\bar{\s}$ of $\s$ to the cylinder is the vector field:
\begin{align*}
\label{eq:dF}
\bar\s(\bar x,\bar y,\bar z)
&=d\psi_{(x,y,z)}(u,v,w) \\
&=(a\bar y+b\bar z)(1-\bar x^2,-\bar x\bar y,-\bar x\bar z)
+(0,a\bar x + c\bar z,b\bar x - c\bar y).
%&=\bigl((1-\bar x^2)(a\bar y+b\bar z),
%-\bar x\bar y(a\bar y + b\bar z) + a\bar x + c\bar z,
%-\bar x\bar z(a\bar y + b\bar z) + b\bar x - c\bar y\bigr).
\end{align*}
Notice that $\bar\s$ extends smoothly across $\partial C$; ie.\ when $\bar x=\pm 1$.  Then $\bar\s(\bar x,\bar y,\bar z)=0$ for $(\bar x,\bar y,\bar z)\in C \cup \partial C$ if and only if the following non-linear system is satisfied:
\begin{align*}
0 &=(1-\bar x^2)(a\bar y+b\bar z), \\
0 &=a\bar x +c\bar z-\bar x\bar y(a\bar y+b\bar z), \\
0 &=b\bar x-c\bar y-\bar x\bar z(a\bar y+b\bar z).
\end{align*}
Note first that the constraint $\bar y^2+\bar z^2=1$ ensures that solutions exist only if $a^2+b^2\neq0$.  The option $a\bar y+b\bar z=0$ yields solutions:
$$
(\bar x,\bar y,\bar z)=\frac{\pm1}{\sqrt{a^2+b^2}}(c,b,-a).
$$
If $\lambda=0$ then $a^2+b^2=c^2$ and the solutions reduce to: 
$$
\pm(1, b/c, -a/c)\in \partial C,
$$ 
one on each component.  If $\lambda>0$ then $a^2+b^2>c^2$ so $|c|/\sqrt{a^2+b^2}<1$ and the solutions lie in $C$; they correspond to:  
$$
\frac{\pm1}{\sqrt{a^2+b^2-c^2}}(c,b,-a)\in H^2_1.
$$
Finally if $\lambda<0$ then $|c|/\sqrt{a^2+b^2}>1$ so there are no solutions on the closed cylinder.  The option $\bar x^2=1$ yields the previously obtained solutions in $\partial C$.
\end{proof}

\medskip
\begin{prop}\label{prop:killh21}
Let $\s$ be a Killing field on $H^2_1$. If $\lambda<0$, $\lambda=0$, $\lambda>0$, respectively, then $\s$ is congruent to the Killing field whose linear extension has the following normal form, respectively:
$$
\left( \begin{array}{ccc}
0 & 0 & 0 \\
0 & 0 & \sqrt{-\lambda} \\
0 & -\sqrt{-\lambda} & 0 \end{array} \right),
\qquad
\left( \begin{array}{ccc}
0 & 1 & 0 \\
1 & 0 & 1 \\
0 & -1 & 0 \end{array} \right),
\qquad
\left( \begin{array}{ccc}
0 & \sqrt{\lambda} & 0 \\
\sqrt{\lambda} & 0 & 0 \\
0 & 0 & 0 \end{array} \right).
$$
\end{prop}

\begin{proof}
We give the argument for $\lambda<0$ (which is the case directly relevant to Theorem \ref{thm:2dkf}); the other cases are similar.  If $a^2+b^2=0$ then the matrix is already in normal form.  Otherwise, consider the infinitesimal isometry $\rho$ of $H^2_1$ whose linear extension has matrix:
$$
\left( \begin{array}{ccc}
0 & \alpha & \beta \\
\alpha & 0 & 0 \\
\beta & 0 & 0 \end{array} \right),
$$
where $\alpha={b}/{\sqrt{a^2+b^2}}$ and $\beta={-a}/{\sqrt{a^2+b^2}}$; thus $\a^2+\b^2=1$.  After solving an appropriate system of first order linear ODE (whose details we omit), the flow of $\rho$ is the restriction to $H^2_1$ of the following linear flow on $\R^3_2$:
$$
\Phi_t
=\begin{pmatrix}
\cosh t  & \a\sinh t & \b\sinh t \\
\a\sinh t & \b^2+\a^2\cosh t & \a\b (\cosh t-1) \\
\b\sinh t & \a\b(\cosh t-1) & \a^2+\b^2\cosh t 
\end{pmatrix}.
$$
If $c>0$ and the parameter $t_0$ is chosen such that $\cosh(t_0)=c/c_0$ where $c_0=\sqrt{-\lambda}$ then:
$$
\Phi_{t_0}
=\frac{1}{c_0}\begin{pmatrix}
c  & b & -a \\
b & c_0+b^2C & -abC \\
-a & -abC & c_0+a^2C 
\end{pmatrix}
$$
where $C=(c-c_0)/(a^2+b^2)$.  Then after some further computation:
$$
\Phi_{-t_0}A\/\Phi_{t_0} 
=\begin{pmatrix}
0 & 0 & 0 \\
0 & 0 & c_0 \\
0 & -c_0 & 0 
\end{pmatrix},
$$
which when restricted to $H^2_1$ yields the desired congruence.
\end{proof}

\medskip
\begin{thm}\label{thm:harmkill}
Let $M$ be a pseudo-Riemannian quadric of sectional curvature $\e=\pm1$, other than the Riemannian $2$-sphere.  Then up to congruence there exists a unique harmonic Killing field $\s$ on $M$, which is the restriction of one of the following matrices:
$$
\begin{pmatrix}
0 & 1& 0 \\
\e & 0 & 0 \\
0 & 0 & 0 
\end{pmatrix},
\qquad
\begin{pmatrix}
0 & 0 & 0 \\
0 & 0 & 1 \\
0 & \e & 0 
\end{pmatrix},
\qquad
\begin{pmatrix}
0 & 1 & 0 \\
-1 & 0 & 0 \\
0 & 0 & 0 
\end{pmatrix},
$$
according as $M=S^2_2$ or $H^2_0$, $M=S^2_1$ or $H^2_1$ or $M=H^2_0$, respectively. In all cases the metric parameters of $\s$ are $(3,-1/2)$.
\end{thm}

\bigskip    
\section{Para-K{\"a}hler twisted anti-isometries} 
\label{sec:PK}  
    
We recall \cite{CFG} that an {\sl almost para-Hermitian structure\/} on a pseudo-Riemannian manifold $(M,g)$ is a skew-symmetric $(1,1)$-tensor field $J$ satisfying $J^2=1$.  The existence of such a structure forces $(M,g)$ to be of even dimension and neutral signature.  If in addition $\nabla J=0$ then $J$ is {\sl para-K\"ahler.}  Because almost para-Hermitian structures are anti-isometric in the following sense:
$$
g(JX,JY)=-g(X,Y),
$$
a para-K\"ahler twisted harmonic vector field need not be harmonic; see \cref{eg:pktwist} below.  However combining an anti-isometry $\v$ with a para-K\"ahler twist $J$ rectifies this problem, for both $\v$ and $J$. 

\medskip
\begin{prop}\label{prop:pktwist}
Let $(M,g,J)$ be a para-K{\"a}hler manifold and $\varphi\colon(M,g)\to(N,h)$ an anti-isometry.  If $\s$ is a harmonic vector field on $(M,g)$ then the push-forward $\varphi_*(J\sigma)$ is a harmonic vector field on $(N,h)$, with the same metric parameters.         
\end{prop}   

\begin{proof}
Abbreviating $\tilde{\sigma}=\varphi_*(J\sigma)$, we have:
\begin{align*}
h(\tilde{\sigma},\tilde{\sigma}) 
&= h(\varphi_*(J\sigma),\varphi_*(J\sigma))
=-g(J\sigma,J\sigma)
=g(\sigma,\sigma).
\end{align*}
Thus $\tilde F=F$.  Since $d\v$ and $J$ are parallel, all remaining pieces of the Euler-Lagrange vector field $\tau_{p,q}(\s)$ (see \cref{thm:pqharm}) are invariant, and we conclude that:
\begin{equation*}
\tau_{p,q}(\tilde\s)=\v_*(J\,\tau_{p,q}(\s)).
\qedhere
\end{equation*}
\end{proof}

\medskip
\begin{rem}\label{rem:pktwist}
A similar result holds if $\v$ is an anti-isometry into a para-K\"ahler manifold: if $\s$ is a harmonic vector field on the domain then $J(\v_*\s)$ is a harmonic vector field on the codomain.
\end{rem}

\medskip
We recall also that a vector field $\s$ on $(M,g)$ is said to be {\sl closed conformal\/} if $\s$ is conformal and its metrically dual $1$-form is closed \cite{Cam}.  By \cite{KK, KR} closed conformal vector fields are characterised by the following generalisation of \cref{prop:cgf}\,(3):
\begin{equation}
\nab X\s=\psi\/X,
\label{eq:ccvf}
\end{equation}
for some smooth function $\psi\colon M\to\R$, where necessarily $n\psi=\diverge\s$.

\medskip
\begin{prop}\label{prop:ccfpkkf}
Let $\sigma$ be a closed conformal vector field on a para-K\"ahler manifold $(M,g,J)$. Then $J\sigma$ is a Killing field.
\end{prop} 

\begin{proof}
Since $J$ is para-K\"ahler:
\begin{align*}
g(\nab X (J\s),Y) + g(X,\nab Y (J\s)) 
&=g(J\,\nab X\s,Y)+g(X,J\,\nab Y\s) \\
&=g(J(\psi X),Y)+g(X,J(\psi Y)),
\quad\text{by \eqref{eq:ccvf}} \\
&=-\psi g(X,JY)+\psi g(X,JY)=0.
\end{align*}
Hence $J\s$ is Killing.
\end{proof}

\medskip
Every oriented $2$-dimensional pseudo-Riemannian manifold of neutral signature  admits a unique para-K\"ahler structure that is compatible with the orientation in the following sense.  The null vectors $L\subset TM$ may be written $L=L_1\cup L_2$ where $L_1,L_2\subset TM$ are distinct line sub-bundles, labelled such that if $(A,B)$ is a positively oriented local tangent frame with $A\in L_1$ and $B\in L_2$ then $A+B$ is space-like (which implies $A-B$ is time-like).  Then define:
$$
JA=A,
\qquad
JB=-B.
$$ 
It is easily checked that $J$ is para-K\"ahler.  In particular, if $M$ is a neutral quadric then it follows from \cref{prop:ccfpkkf} that $\s\mapsto J\s$ yields a linear involutive isomorphism between the Killing and conformal gradient fields on $M$, since both spaces have the same dimension (namely, $3$).  Hence by \cref{prop:pktwist}, if $\v$ is the canonical anti-isometry from $H^2_1$ to $S^2_1$ then $\s\mapsto \v_*(J\s)$ yields a bijection between the unique congruence class of harmonic conformal gradient fields (resp.\ Killing fields) on $H^2_1$ and the congruence class of harmonic Killing fields (resp.\ conformal gradient fields) on $S^2_1$.  These classes are also bijectively equivalent via the correspondence of \cref{rem:pktwist}, using the para-K\"ahler structure of $S^2_1$.  However since $\v$ is para-holomorphic the two bijections are in fact the same. 

\medskip
\begin{eg}\label{eg:pktwist}
As an explicit example, let $\s$ be the conformal gradient field on $H^2_1$ with pole vector $(0,0,1)$, which is harmonic by \cref{thm:cgf}.  Then:
$$
(J\s)(x,y,z)=(y,x,0),
$$
which although Killing (\cref{prop:killh21}), with the same zeros $(0,0,\pm1)$ as $\s$, is not harmonic (\cref{thm:harmkill}); indeed, the harmonic Killing fields on $H^2_1$ have no fixed points.  From \eqref{eq:antisom} the canonical anti-isometry from $H^2_1$ to $S^2_1$ is:
$$
\v(x,y,z)=(z,x,y),
$$
and the push-forward of $J\s$ by $\v$ is the vector field:
$$
(x,y,z)\mapsto d\v((J\s)(\v\inv(x,y,z))
=d\v((J\s)(y,z,x))
=d\v(z,y,0)=(0,z,y),
$$
which by \cref{thm:harmkill} is harmonic, with zeros $(\pm1,0,0)$.
\end{eg}

\end{document}